\documentclass[11pt,a4paper]{article}
\pdfoutput=1

\usepackage[english]{babel}
\usepackage{a4wide}
\RequirePackage{ucs}
\RequirePackage[utf8x]{inputenc}
\RequirePackage[T1]{fontenc}

\RequirePackage{amsthm,amsmath}
\usepackage{multirow}
\usepackage{bigdelim}
\RequirePackage{amssymb,amstext}
\RequirePackage{amsfonts}
\RequirePackage{amsbsy}
\RequirePackage{mathrsfs}
\RequirePackage{bm}
\RequirePackage{bbm}
\usepackage{easybmat}
\usepackage{enumerate}
\usepackage[all]{xy}
\usepackage{tikz}
\usetikzlibrary{decorations.pathmorphing}
\usepackage{multirow,bigdelim}
\usepackage{MnSymbol}
\usepackage[round]{natbib}
\usepackage{amsthm}
\usepackage{theoremref}

\usetikzlibrary{positioning,arrows}
\tikzset{
  state/.style={circle,draw,minimum size=6ex},
  arrow/.style={-latex, shorten >=1ex, shorten <=1ex}}

\newcommand{\E}{\mathbb{E}}
\newcommand{\R}{\mathbb{R}}
\newcommand{\cP}{\mathcal{P}}
\newcommand{\N}{\mathbb{N}}
\newcommand{\BR}{\mathfrak{B}_{\mathbb{R}}}

\newcommand{\Z}{\mathbb{Z}}
\newcommand{\A}{\mathcal{A}}
\newcommand{\bP}{\mathbb{P}}
\newcommand{\cE}{\mathcal{E}}
\newcommand{\kS}{\mathfrak{S}}
\newcommand{\cL}{\mathcal{L}}
\newcommand{\U}{\mathcal{U}}

\newcommand{\D}{\mathcal{D}}

\newtheorem{theorem}{Theorem}[section]

\newtheorem{proposition}[theorem]{Proposition}
\newtheorem{lemma}[theorem]{Lemma}
\newtheorem{corollary}[theorem]{Corollary}
\newtheorem{remark}{Remark}
\newtheorem{example}{Example}

\title{Bayesian nonparametric inference for the M/G/1 queueing systems based on the marked departure process}

\usepackage{authblk}

\author[1]{Moritz von Rohrscheidt\thanks{rohrscheidt@uni-heidelberg.de}}
\author[2]{Cornelia Wichelhaus \thanks{wichelhaus@mathematik.tu-darmstadt.de}}

\affil[1]{Ruprecht-Karls-Universit{\"a}t Heidelberg}
\affil[2]{Technische Universit{\"a}t Darmstadt}

\begin{document}

%%%%%%%%%%%%%%%%%%%%%%%%%%%%%%%%%%%%%%%%%%%%%%%%%%%%%

\maketitle

\begin{abstract}
In the present work we study Bayesian nonparametric inference for the continuous-time
M/G/1 queueing system. In the focus of the study is the unobservable service time
distribution. We assume that the only available data of the system are the marked
departure process of customers with the marks being the queue lengths just after
departure instants. These marks constitute an embedded Markov chain whose distribution
may be parametrized by stochastic matrices of a special delta form. We develop the theory
in order to obtain integral mixtures  of Markov measures with respect to suitable prior
distributions. We have found a sufficient statistic with a distribution of a so-called
S-structure sheding some new light on the inner statistical structure of the M/G/1 queue.
Moreover, it allows to update suitable prior distributions to the posterior.  Our
inference methods are validated by large sample results as posterior consistency and
posterior normality. 
\end{abstract}


\section{Introduction}\label{intro}

Considering the rapidly growing areas of \textit{artificial intelligence} and \textit{Industrie 4.0}, strengthening the interface of the fields of operations research and statistics becomes more and more important and contributes to a highly important field of new studies. In this paper we consider statistical inference in the Bayesian paradigm for the continuous-time $M/G/1$ queue. Thereby, on the one hand, the $M/G/1$ queueing model serves as a stochastic tool for modeling a wide range of real world situations as e.g. customers arriving to a check out in a supermarket, cars arriving to a traffic jam, information arriving to a transmitting server, goods arriving to a storage, energy arriving to an electrical grid and many others. On the other hand all these situations are subjected to random fluctuation which have to be modeled stochastically and which need to be inferred statistically. Since people setting up new technology usually have gained prior knowledge from the past or from their expertise, we are convinced that the Bayesian approach to statistics fits the statistical modeling of those situations best. Notice that already Dennis Lindley who contributed a lot to the theory of queueing processes on the one side and was a great advocate of the Bayesian interpretation of statistics on the other hinted to the power of linking these fields, c.f. his foreword in \citet{de1974theory}. Since then a lot of work was done which created a vivid area of research which nowadays is better known as \textit{Bayesian queues} and which the present work gets in line with. It is worth mentioning some of the milestones of this theory. Seminal works were given by \citet{armero1985bayesian} and \citet{mcgrath1987subjective1, mcgrath1987subjective2}. Based on the explicit solvability of the parametric $M/M/1$ model \citet{armero1985bayesian} obtains the posterior law of the traffic coefficient as well as the posterior predictive distribution of the waiting time of customers and the number of customers being present in a stable queue, respectively. \citet{mcgrath1987subjective1, mcgrath1987subjective2} refine aforementioned work by considering a finite waiting room as well as different observational setups also taking into account dependence among the arrival and service time processes. These works form the base for further studies of queueing systems with Markovian characters under several generalizations as e.g. \citet{armero1994prior,armero1998inference, armero2004statistical, armero2006bayesian} and \citet{ausin2007bayesian, ausin2008bayesian}. But also studies for non-Markovian queueing models were undertaken as in \citet{wiper1998bayesian, ausin2007bayesian, ausin2008bayesian} who consider generalizations of the interarrival time process. However, while the assumption of a Poisson process as arrival stream is most often well suited, the assumption of exponential service times is not in most real world applications. Therefore theory was further extended in order to be able to model the service time distribution more flexibly which in turn lead to larger classes of parametric and semi-parametric distributions. See e.g. \citet{insua1998bayesian} who employ parametric Erlang and hyperexponential distributions and \citet{ausin2004bayesian} who exploit the well known probabilistic result that the class of phase-type distributions forms a dense set within the space of all probability distributions  $[$cf.~\citet[p.84]{asmussen2008applied}$]$ in order to design the service time distribution semi-parametrically and inferred subsequently using MCMC procedures. However, assuming semi-parametric classes of distributions are still narrowing the modeling of many real world situations. Hence, nonparametric inference for queueing systems is desirable which leads one to investigate the $M/G/1$ continuous-time queueing model statistically. Inference for the  discrete-time analog of $M/G/1$, namely the $Geo/G/1$ model was considered in \citet{conti1999large}. In this work the author takes the random service time distribution to be sampled from a Dirichlet process prior and develops nonparametric inference for the waiting time distribution which is then validated by large sample results about the respective posterior law. The conceptual basis of \citet{conti1999large} is also shared by \citet{rohrscheidt2017bayesian} who move on to the even more involved continuous-time framework. In that work the authors assign a neutral-to-the-right beta-Stacy prior, which can be viewed a direct generalization of the Dirichlet process prior $[$see e.g. \citet{phadia2015prior}$]$, to the law of the random service time distribution and infer the posterior law of several queueing characteristics using large sample techniques. Thereby the observational setup consists of observations of the arrival stream as well as the service process. Moreover, assuming independence of those processes, results about the joint posterior can be obtained regarding the updates of these two processes separately. However, in many situations neither the arrival stream is directly observable nor is the service process. In those cases the only observations one has access to are constituted of the departure process. But since the interdeparture times alone are non-informative for an $M/G/1$ queueing system in steady state, additional information has to be taken into account. In the present work we assume that this information is encoded in form of marks the departure point process carries. Yet, a similar independence assumption among the interdeparture times and the marks cannot be argued. This fact gives rise to study the inner structure of the $M/G/1$ system from the viewpoint of a subjectivistic probabilist. Roughly speaking this means that Markov laws governed by stochastic matrices possessing a certain shape need to be mixed up by a probability distribution which then forms the prior of the system. Furthermore, a sufficient statistic needs to be found which can be used to update this prior to the posterior. These two issues are addressed in this paper. \\

The paper is organized as follows. In section 2 the details of the underlying queueing model are introduced as well as some basic assumptions which enable us to infer characteristics of the system as e.g. the distributions of the service times and the waiting times based on mere observations of the marked departure process. Section 3 deals with the issue of assigning suitable priors for the random interdeparture time distribution and the random law which governs the marks of the departure process, respectively. Therefor a sufficient statistic is given which summarizes the information properly in order to build the posterior and some examples are given to strengthen the comprehension of this statistic. Section 4 is devoted to inference. This encompasses the development of estimators for the service time distribution and for the distributions of further queueing characteristics, respectively, based on some functional relationship of their Laplace-Stieltjes transforms. Moreover those inference methods are supported by results about posterior consistency as well as posterior normality.

%%%%%%%%%%%%%%%%%%%%%%%%%%%%%%%%%%%%%%%%%%%%%%%%%%%%%

\section{Description of the queueing model and statistical assumptions}\label{model}

The model under consideration is the basic $M/G/1$ queueing model, where $M$ refers to the input stream modeled as a Poisson process with homogenous rate $\lambda>0$, $G$ to a general probability distribution for the service times provided by $1$ reliable server. The policy according which customers are served is assumed to be first in first out (FIFO). We assume throughout, that the queueing system has already reached its steady state, which makes the process of successive  waiting times of customers a stationary process. Letting the traffic intensity be the ratio of the mean service time and the mean interarrival time, a sufficient and necessary condition for the system to reach equilibrium is that the traffic intensity is strictly smaller than one. Equivalently, the probability to find the system unoccupied has to be strictly greater than zero, see e.g. \citet{asmussen2008applied}.
We recall some interesting features of the $M/G/1$ system which motivate our assumptions on the statistical observations which are stated later in this section. \\

The first feature is the PASTA property. PASTA, abbreviating \textit{Poisson-Arrivals-See-Time-Averages}, means that the probability to find the (stable) system in a certain state at an arbitrary time point is equal to the probability that any arriving customer finds the system in this certain state. Note that PASTA strongly depends on the lack of anticipation assumption (LAA) of the arrival stream. Roughly speaking,  LAA means that future increments of the arrival process are not affected by the present state $s\in\N_0$ of the queueing system. This assumption is clearly met by the Poisson process governing the arrivals of customers. For more details on LAA and PASTA see \citet{wolff1982poisson}. Thus, from a statistical point of view, if one is interested in making inference for the queueing system based on consecutive observations, there is no difference in collecting respective data at instances of arrivals or arbitrarily chosen instances. The second property of stable $M/G/1$ systems is the so called level crossing law (LC). LC says that the limiting fraction of arriving customers seeing $s\in\N_0$ customers in front equals the limiting fraction of departing customers leaving $s$ customers behind. See \citet{brill2008level} for an exhaustive treatment of level crossing laws. Thus, as long as the queueing system has attained equilibrium, we conclude by PASTA and LC that the law of the system size at arbitrary instances of time equals the law of the system size at departure times of customers. This will play a crucial role in our statistical treatment. The third property we briefly review is perhaps even more unexpected at first glance. It says that in steady state the departure process from the queue, i.e. the continuous time stochastic process given by the random departure time points, is a Poisson process with the same intensity as the arrival process. This relies on the fact that any stable (in the literature in abuse of notation often called \textit{ergodic}) birth death process is time-reversible, c.f. \citet[page 114]{asmussen2008applied}. Thus, the reversed system evolves as the original $M/G/1$ which in turn implies a Poisson process for the departure stream with the same intensity as the arrival process. However, the latter property affects our statistical viewpoint on the system. Since we assume that we observe a system in steady state, solely observing the departure stream is completely non-informative with respect to the unknown service time distribution. Hence, we will need an additional observation that provides enough information to learn about the random service time distribution.\\

To motivate the choice of these additional observations, we briefly review some of the theory about the $M/G/1$ model which requires us to introduce some notation. Let's pick randomly a customer which will refer to as customer $=$. Further, for an integer $n$, let $S_n$ denote the random service time of the $n^{th}$ customer, $T_n$ the instance of time at which customer $n$ departs from the system and $D_n:=T_{n+1}-T_n$ be the interdeparture time between the $n^{th}$ and the $(n+1)^{st}$ customer. Furthermore, let $A_n$, denote the interarrival time between the $n^{th}$ and $(n+1)^{st}$ customer and $A_{S_n}$ be the number of customers entering the system during the service of customer $n$. Since all of the above random variables do not depend on a specific customer as long as the system is in steady state, we feel free to write them without any index. Note that by above mentioned properties of $M/G/1$, we have $\mathcal{L}[A]=\mathcal{L}[D]=\mathcal{E}(\lambda)$, where $\mathcal{L}$ stands for the the law of some random quantity and $\mathcal{E}(\lambda)$ for the exponential distribution with rate $\lambda>0$. Let $\cP(\mathcal{S})$ denote the space of all probability measures on some Polish space $\mathcal{S}$, $G\in \cP(\R_+)$ be the general distribution of the service times and $(\Omega,\A,\bP)$ denote an arbitrary probability space which represents the domain of all random quantities being involved. Now, by the law of total probability one has for $k\in\N_0$

\begin{align*}
\bP(A_S=k)=\int\limits_0^{\infty} \bP(A_S=k|S\leq t) G(dt)= \frac{1}{k!}\int\limits_0^{\infty} e^{-\lambda t} (\lambda t)^k G(dt),
\end{align*}
which in turn yields
\begin{align}\label{LST}
a(z):=\sum\limits_{k=0}^{\infty} z^k \bP(A_S=k)=\int\limits_0^{\infty} e^{-\lambda t} \left[\sum\limits_{k=0}^{\infty} \frac{(z \lambda t)^k}{k!}\right] G(dt)=\int\limits_0^{\infty} e^{-\lambda[1-z]t} G(dt)=:g(\lambda\left[1-z\right]), 
\end{align}
using the monotone convergence theorem.
Hence, a functional relationship of the probability generating function (p.g.f.) $a(\cdot)$ of $A_S$, the departure rate $\lambda$ and the Laplace-Stieltjes transform (LST) $g(\cdot)$ of the service time distribution $G$, is obtained. This functional relationship will be the starting point for our inferential analysis.\\

We continue by regarding $A_S$ with respect to possible observations of the system.
Therefor let $\{N(t)\}_{t\in \R}$ denote the stochastic process with state space $\N_0$ that describes the number of customers in the system at time $t$. It is plain that in general $N(t)$ is not a Markov process since the entire history of this process is informative for a next outcome. The only situation in which $N(t)$ can be considered Markovian is when $G=\mathcal{E}$, which is due to the memorylessness property of the exponential distribution. However, a ``sub-process'' of $N$ can be found that is a Markov chain and thus makes $N$ a semi-Markov process. Call $N(t)$ a semi-Markov process with state space $\N_0$ according to the following construction. Assume there is a stochastic kernel $\kappa: \N_0\times (\mathfrak{B}_{\N_0} \otimes \BR) \rightarrow [0,1]$ and a discrete-time stochastic process $(\tau_n)_{n\in\Z}$ such that $\bP$-a.s. $\tau_n<\tau_{n+1}$ and $N(t)=c_{n+1}$, for all $t\in[\tau_n,\tau_{n+1})$ and for some sequence of states $(c_i)_{i\in \Z}$. Define a two-component discrete-time stochastic process $Y$ by $Y_n:=(c_n, \tau_n-\tau_{n-1})$. Then, $N(t)$ is a semi-Markov process with kernel $\kappa$ if $\bP(Y_{n+1}=y_{n+1}| Y_i; i\leq n)=\bP(Y_{n+1}=y_{n+1}|Y_n)=\kappa(c_n,y_{n+1})$.\\

The process $N(t)$ meets this definition since the system size is constant on intervals where no departure and no arrival occur. The holding times $\tau_n$ then are seen to be the minimum of the residual service time of the customer currently being served (if there is any) and the time elapsing until the next arrival of a new customer. Although, due to the memorylessness property of the exponential distribution, the time until the next arrival is an exponential distribution, the description of the entire holding time distribution is difficult. However, it shall become easier if we circumvent taking into account the residual service time. That means we observe the system at instances at which customers depart from the system, i.e. at instances $(T_n)_{n\in\Z}$. Indeed, it turns out that the process $\{N(T_n),T_n\}_{n\in\Z}$ is a Markov chain, the so called embedded Markov chain. To see this, note that 
\begin{align*}
&N(T_{n+1})=N(T_n)+A_{S_{n+1}}-\left[1-\delta_{N(T_n)}(\{0\})\right],\\
&T_{n+1}=[1-\delta_{N(T_n)}(\{0\})](T_n+S_{n+1})+\delta_{N(T_n)}(\{0\})(T_n+S_{n+1}+I_{n+1}),
\end{align*}
where $I_{n+1}\sim \cE(\lambda)$ reflects the exponentially distributed remainder of the inter-arrival time between the $n^{th}$ and $(n+1)^{st}$ customer, which is independent of $\{(N(T_i),T_i):i<n\}$ as well as $S_{n+1}$ is by assumption. If one focuses solely on the Markov chain $\{N(T_n)\}_n$, above equation gives rise to the stochastic matrix governing the chain $N:=\{N(T_n)\}_n$. Let $M=(m_{ij})_{i,j\in \N_0} \in\kS$ denote the infinite matrix consisting of all probabilities of transitions of the chain $N$ from state $i$ to state $j$, where $\kS$ denotes the space of all infinite stochastic matrices. Then, the probability of having a transition from $i$ to $j$ is given through
\begin{align*}
m_{ij}=
\begin{cases} \bP\left(A_S=j-i+[1-\delta_{i}(\{0\})]\right) & \textit{, if } i-j\leq 1\\
              0 &\textit{, esle }
\end{cases},
\end{align*}
which yields the following form of $M$.
\begin{align*}
M=\begin{pmatrix} 
a_0&a_1&a_2&a_3&a_4&a_5&\cdots\\
a_0&a_1&a_2&a_3&a_4&a_5&\cdots\\
0&a_0&a_1&a_2&a_3&a_4&\cdots\\
0&0&a_0&a_1&a_2&a_3&\cdots\\
0&0&0&a_0&a_1&a_2&\cdots\\
\vdots&\ddots&\ddots&\ddots&\ddots&\ddots&\ddots
\end{pmatrix},
\end{align*}
where $a_i=\bP(A_S=i)$ for $i\in\N_0$.\\

Note that this stochastic matrix is a member of a larger family of stochastic matrices, the so called $\Delta$-matrices which were introduced in \citet{abolnikov1991markov}. To be more precise, $M$ is a positive homogenous $\Delta_{1,1}$, where the first $1$ indicates that $p_{ij}=0$ for all $i-j>1$ and the second that this holds for all rows $i>1$. Homogeneity thereby means that $m_{ij}=a_{j-i+1}$ for all $i>1$ and positivity that $a_{j-i+1}>0$ whenever the index is non-negative. Notice that positivity is given since the number of customers who enter the system during a service time is not bounded, even if its probability might decay at a fast rate. Moreover, \citet{abolnikov1991markov} study the (Markovian) ergodicity of Markov chains governed by some $\Delta$-matrix and obtain necessary and sufficient conditions for that in terms of the p.g.f. of the discrete distribution $(a_i)_{i\in\N_0}$. For the stochastic matrix $M$ considered here, it is easy to see that it is irreducible and aperiodic since 
\begin{align*}
M^k=
\left(
\begin{array}{*8c}
x & x& x& x& x& x & \cdots &\rdelim){9}{4em}\\
 x & x &x & x&x &x & \cdots& \\
 x & x& x &x &x&x & \cdots &\rdelim\}{3}{1cm}[$(k-1)$-times]\\
\vdots&\vdots&\vdots&\vdots&\vdots&\vdots&\cdots&\\
x&x&x&x&x&x&\cdots&\\
0 &x &x &x &x & x &\cdots\\
0 &0 &x &x &x & x &\cdots\\
0 &0 &0 &x &x & x &\cdots\\
\vdots&\ddots&\ddots&\ddots&\ddots&\ddots&\ddots&
\end{array}\right. ,
\end{align*}
where an entry $x$ means an entry strictly greater than zero. Moreover, Theorem 3.4. in \citet{abolnikov1991markov} shows that the Markov chain $N$ governed by above matrix $M$ is positive recurrent if and only if $\frac{\partial a(z)}{\partial z}_{|_{z=1}}=:a'(1)<1$. Since $a'(1)=\lambda \E[S]=\E[S]/\E[I]=:\rho$ by \eqref{LST}, this requirement becomes the common condition in queueing theory for stability of the $M/G/1$ system. It is well known from the theory of Markov chains that for $M$ of above form with $a'(1)<1$ there exists a unique $M$-invariant distribution $p\in\cP(\N_0)$, i.e. $p$ is a stochastic row vector with all entries strictly positive such that $p M=p$. Letting $\pi(\cdot)$ denote the p.g.f. of $p$, this yields
\begin{align*}
\pi(z)=a(z)\frac{(1-z)(1-a'(1))}{a(z)-z},
\end{align*}
c.f. \citet{nelson2013probability}. All of the aforementioned facts about $M/G/1$ lead us to the following assumption on the accessible observations. 
\begin{align}\label{O}
\textsf{Data is collected from observations of the stochastic process } \{N(T_n),T_n\}_{n\in \N}. \tag{O}
\end{align}
Assumption \eqref{O} means that we observe the instants of time at which the consecutively served customers depart from the system as well as the number of customers they leave behind in the system. Thus, one can imagine the $M/G/1$ system as a perfect black box. The only thing to which we have access is the departure process being a marked point process. Due to time-reversability, this point process is a marked Poisson process with rate $\lambda>0$ and with marks consisting of the system size at departure time points. Notice that the inter departure times do not directly depict the service times since they are possibly corrupted by idle times under the promises that $\rho<1$.\\

We are interested in making Bayesian statistical inference for several characteristics of the system on basis of these observations. The characteristics are e.g. the unknown service time distribution, waiting time distribution and the distribution of the busy and idle times.

%%%%%%%%%%%%%%%%%%%%%%%%%%%%%%%%%%%%%%%%%%%%%%%%%%%%%%

\section{Prior assignments}\label{prior}

In this section we assign prior distributions to the laws of the random quantities which appeared in section \ref{model}. The observations \eqref{O} from the previous section already indicate that the issue is twofold. Firstly, we can extract the interdeparture times $D_n:=T_n-T_{n-1}$ from the observations. As argued in the previous section, these can be viewed as independent and exponentially distributed random quantities with rate $\lambda$, leading to a parametric inference problem. On the other hand, the marks $\{N(T_n)\}_n$ were argued to form a Markov chain with stochastic $\Delta_{1,1}$-matrix, leading to a non-parametric inference problem. However, from a subjectivist Bayesian viewpoint, data is not independent or Markov, respectively. As one needs to express one's uncertainty about the law of the random quantities, data is rather assumed to be exchangeable or partially exchangeable, respectively. These two issues are now studied in detail separately.

\subsection{Prior for the interdeparture time distribution}
For a sequence of random quantities $X_{[1:\infty]}:=(X_n:\Omega \rightarrow \R)_{n\in\N}$, say that $X_{[1:\infty]}$ is exchangeable if for any $n\in\N$ and any permutation $\sigma$ of $n$ elements it holds that 

\begin{align*}\label{E}
\cL\left[X_1,\dotsc X_n\right]=\cL\left[X_{\sigma(1)},\dotsc, X_{\sigma(n)}\right]. \tag{E}
\end{align*}

Then, the celebrated de Finetti theorem for Polish spaces $[$see \citet{hewitt1955symmetric}$]$ gives rise to a prior distribution $\mu\in\cP(\cP(\R))$. To be more precise, $X_{[1:\infty]}$ is exchangeable if and only if there is a unique measure $\mu$ such that
\begin{align*}
\bP(X_i\in A_i; i=1,\dots,n)=\int\limits_{\cP(\R)} \prod\limits_{i=1}^n P(X_i\in A_i) \mu(dP),
\end{align*}
where $A_i\in\BR$ for all $i=1,\dots,n$, for all $n\in\N$. Note that $\cP(\R)$ is a Polish space if it is assumed to be equipped with the sigma field induced by weak convergence of measures, see \citet{kechris1995descriptive,billingsley2013convergence}. The introduced sigma field on $\cP(\R)$ is the smallest one with respect to which the mappings $P\mapsto P(A)$ are measurable for all sets $A\in\BR$. Notice that in Bayesian statistics the assertion of the de Finetti theorem is frequently interpreted as random variables being conditionally i.i.d. given their common distribution $P$. This indicates that exchangeable data $X_{[1:\infty]}$ are not independent unconditionally on $P$. Note the conceptual difference to the frequentist approach. Indeed, a result of \citet{kingman1978uses} shows that in general the exchangeable random variables are positively correlated, a fact that makes Bayesian statistics a theory of inferential prediction and forecasting. 
Moreover, \citet{diaconis1985quantifying} show that the additional symmetry assumption 
\begin{align*}\label{exp}
\bP(X_i\in A_i; i=1,\dots,n)=\bP(X_i\in C_i; i=1,\dots,n),\tag{$\exp$}
\end{align*}
for all $A_i\in\BR^+$ and constants $c_i\in\R$ such that $\sum_{i=1}^nc_i=0$ and $C_i=c_i+A_i=\{c_i+a: a\in A_i\}\in\BR^+$, $i=1,\dots,n$, shrinks the support of the mixing measure $\mu$ to the set of all exponential distributions denoted as $\mathcal{E}$. Thus, under the assumption \eqref{E} and \eqref{exp}
\begin{align*}
\bP(X_i\in A_i; i=1,\dots,n)=\int\limits_{\cE} \prod\limits_{i=1}^n P(X_i\in A_i) \mu(dP)=\int\limits_{\R^+} \prod\limits_{i=1}^n \left[\int_{A_i} \lambda e^{-\lambda x_i} dx_i \right]\tilde{\mu}(d\lambda),
\end{align*}
where the natural parametrization $\tilde{} : \cE\rightarrow \R^+; P\mapsto (\int_{\R}xdP)^{-1}$ is used. Furthermore, \citet{diaconis1985quantifying} show that if the random variables $X_{[1:\infty]}$ fulfill a certain regression equation given through $\E[X_2|X_1]=\alpha X_1+\beta$ for some constants $\alpha,\beta\in\R$, then the mixing measure $\tilde{\mu}$ is a gamma distribution.\\

Since the stable $M/G/1$ system gives rise to i.i.d. exponential interdeparture times, from the subjectivist Bayesian viewpoint, we make  the following assumption on the observations $D_{[1:n]}$. The data $D_{[1:n]}$ are assumed to be the first $n$ projections of an infinite exchangeable sequence of random variables $D_{[1:\infty]}$. Moreover, 
\begin{align*}
D_{[1:\infty]} | \lambda &\sim \cE(\lambda)^{\bigotimes_{\N}}\\
\lambda &\sim \Gamma(a,b)
\end{align*}
Thereby, $\Gamma(a,b)$ denotes the gamma distribution with parameters $a,b>0$.
Since the gamma distribution is well known to be a conjugate prior for the rate of exponential distributions, it follows easily that
\begin{align*}
\lambda | (a,b),D_{[1:n]} &\sim \Gamma\left(a+n,b+\sum_{i=1}^n D_i\right)\\
\E_{\Gamma}\left[\lambda | D_{[1:n]},(a,b)\right]&= \frac{a+n}{b+\sum_{i=1}^n D_i}\\
f\left(a_{n+1}|D_{[1:n]},(a,b)\right)&=\frac{(a+n)\left(b+\sum_{i=1}^n D_i\right)^{a+n}}{\left(b+\sum_{i=1}^n D_i+a_{n+1}\right)^{a+n+1}}\\
\E\left[D_{n+1}|(a,b),D_{[1:n]}\right]&=\frac{1}{a+n-1}\sum_{i=1}^n D_i +\frac{b}{a+n-1}.
\end{align*}
Note that the latter equation again reflects the learning process which is not directly available in the frequentistic approach. For $n=1$ this is given by $\E[D_2|D_1,(a,b)]=1/a D_1 + b/a$ as required in above assumption on the shape of the prior.

\subsection{Prior for the law of the embedded Markov chain}

Eliciting a prior distribution for the law of the marks $\{N(T_n)\}_n$ of the marked departure process described in section 2 is a more involved task. However, \citet{freedman1962invariants, diaconis1980finetti} show hat an analogue of the de Finetti theorem holds for laws of Markov chains. Let $Y_{[1:\infty]}:\Omega \rightarrow \N_0^{\N}$ be a discrete-time stochastic process with state space $\N_0$. $Y_{[1:\infty]}$ is called partially exchangeable if for all $n\in\N$
\begin{align}\label{PE}
\cL\left[Y_{[1:n]}|t_n(Y_{[1:n]})=r\right]=\U_{t_n^{-1}(r)} \tag{PE}.
\end{align}
Thereby, $Y_{[1:n]}=(Y_1,\dots,Y_n)$, $\U_{B}$ denotes the uniform law on a discrete set $B$ and $t:=(t_n)_{n\in\N}$ is a statistic, i.e. a family of measurable mappings, defined by
\begin{align*}
&(i_1,i_2,\dots,i_n)\mapsto t_n((i_1,i_2,\dots,i_n)):=\left(i_1,T\right),\\
\end{align*}
where $T=(t_{rs})_{r,s\in\N_0}$ is the transition count matrix defined by $t_{rs}:=\#\{j: (i_j,i_{j+1})=(r,s), j=1,\dots,n-1 \}$.
Condition \eqref{PE} is a another way to state that the law of the process $Y_{[1:\infty]}$ is summarized by the statistic $(t_n)_n$, see Freedman (1962). If in addition to \eqref{PE} recurrence holds for the process $Y_{[1:\infty]}$, i.e. $\bP\left(\limsup\limits_{n\rightarrow\infty} \{Y_n=Y_1\}\right)=1$, then results of \citet{diaconis1980finetti} show that there is a unique measure $\mu\in\cP(\cP(\N_0)\times \kS)$ such that for all $n\in\N$
\begin{align*}
\bP(Y_i=y_i; i=1,\dots,n)=\int\limits_{\cP(\N_0) \times \kS} \nu_{y_1} \prod\limits_{i=1}^{n-1} p_{y_i,y_{i+1}} \mu(d\nu,dP).
\end{align*}
Thus, any law of a recurrent process fulfilling \eqref{PE} is a mixture of Markov measures. Note that recurrence is necessary to exclude pathologies from the mixture, c.f. \citet[Example (19)]{diaconis1980finetti}. However, an earlier result appeared in \citet{freedman1962invariants} for measures being shift-invariant, a property which clearly implies recurrence. Therefor, he showed an even more general result for stationary probabilities being summarized by a so called $S$-structure statistic and noted that the family $(t_n)_n$ has $S$-structure. Since the $M/G/1$ system is assumed to be in equilibrium, we feel that this approach is appropriate for the embedded Markov chain of the underlying queueing system. Roughly, this implies for above de Finetti-style theorem for Markov chains that it suffices to regard the mixing measure $\mu$ having support $\kS$ since the corresponding invariant distribution $\nu$ is uniquely determined by the stochastic matrix $P$ $[$see e.g. \citet{chung1967markov} or \citet{freedman1983markov}$]$ .\\

In analogy to the previous subsection, we examine under what constraints the support of $\mu\in\cP(\kS)$ can be shrunk to the set of $\Delta_{1,1}$ matrices governing the embedded Markov chains of $M/G/1$ systems. Of course, $\mathcal{L}(N)$ is summarized by the family of mappings $(t_n)$ since it is Markov. Yet, $(t_n)$ is not minimal sufficient, a fact that yields a different grouping of data strings of equal length in equivalence classes, see examples below. Clearly, there are two properties of $M$ being not fulfilled for arbitrary infinite stochastic matrices. Namely $\Delta$-shape and homogeneity. They should be reflected in the appropriate statistic summarizing a mixture of laws of embedded Markov chains of $M/G/1$. For any positive integer $n$, let $\D_n:=\left\{(a_i)\in\N_0^{n}: a_{i}-a_{i+1}\leq 1 ,\forall 1\leq i \leq n-1\right\}$ be the $n$-dimensional down-skip-free subspace of $\N_0^n$, $\D_{\infty}\subset \N_0^{\N}$ the down-skip-free sequence space and $\D:=\bigcup_{n\in\N}\D_n$ the collection of down-skip-free strings of any length. Moreover, let $\tau:=(\tau_n)_{n\in\N}$ be a family of measurable mappings, each  $\tau_n$ operating on $\D_n$ through
\begin{align*}
\tau_n : &\D_n \rightarrow \N_0\times \{0,1,\dots,n\}\times \N_0^{\N}\\
&a_{[1:n]}\mapsto \left(a_1, \sum\limits_{k=1}^n \delta_{a_k 0}, I(a_{[1:n]})\right),
\end{align*}
where $I(a_{[1:n]}):=\left(\iota_r(a_{[1:n]})\right)_{r\in\N_0}:=\left(\#\left\{j:a_{j+1}-a_j+(1-\delta_{a_j 0})=r, j=1,\dots,n\right\}\right)_{r\in\N_0}$ are the zero-adjusted increments of the data string $a_{[1:n]}$.
Thus, $\tau_n$ records the length, the initial state, the number of zeros and the increments of the down-skip-free data string $a_{[1:n]}$.
Call two strings $a_{[1:n]},b_{[1:n]}\in\D_n$ $\tau$-equivalent and write $a_{[1:n]} \sim_{\tau} b_{[1:n]}$ if and only if $\tau_n[a_{[1:n]}]=\tau_n[b_{[1:n]}]$.\\

Following examples are given to clarify the meaning of the statistic $\tau$.

\begin{example}
Define the following elements of $\D_9$ by 
\begin{align*}
&a_{[1:9]}=121211100, \qquad b_{[1:9]}=102232101, \qquad c_{[1:9]}=123332100\\
&d_{[1:9]}=100234543, \qquad e_{[1:9]}=121002343, \qquad f_{[1:9]}=102102323\\
&g_{[1:9]}=100234543, \qquad h_{[1:9]}=100002345, \qquad i_{[1:9]}=210101100
\end{align*}
For instance, $g_{[1:9]}$ can be viewed as
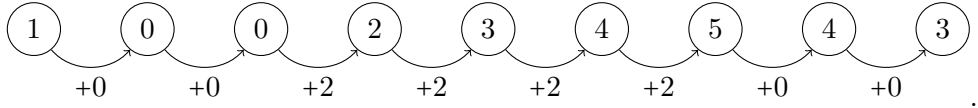
\begin{align*}
\begin{tikzpicture}[node distance=2em]
	\node[state, circle, inner sep=0pt, minimum size=0.7cm] (S1) {$1$};
	\node[state, right=of S1, circle, inner sep=0pt, minimum size=0.7cm] (S2) {$0$};
	\node[state, right=of S2, circle, inner sep=0pt, minimum size=0.7cm] (S3) {$0$};
	\node[state, right=of S3, circle, inner sep=0pt, minimum size=0.7cm] (S4) {$2$};
	\node[state, right=of S4, circle, inner sep=0pt, minimum size=0.7cm] (S5) {$3$};
	\node[state, right=of S5, circle, inner sep=0pt, minimum size=0.7cm] (S6) {$4$};
	\node[state, right=of S6, circle, inner sep=0pt, minimum size=0.7cm] (S7) {$5$};
	\node[state, right=of S7, circle, inner sep=0pt, minimum size=0.7cm] (S8) {$4$};
	\node[state, right=of S8, circle, inner sep=0pt, minimum size=0.7cm] (S9) {$3$};
	 \path[->]
    (S1) edge [bend right=50] node[below] {$+0$} (S2)
		(S2) edge [bend right=50] node[below] {$+0$} (S3)
		(S3) edge [bend right=50] node[below] {$+2$} (S4)
		(S4) edge [bend right=50] node[below] {$+2$} (S5)
		(S5) edge [bend right=50] node[below] {$+2$} (S6)
		(S6) edge [bend right=50] node[below] {$+2$} (S7)
		(S7) edge [bend right=50] node[below] {$+0$} (S8)
		(S8) edge [bend right=50] node[below] {$+0$} (S9);
\end{tikzpicture}.
\end{align*}
This corresponds to the possible path (depending on the times $T_{[1:n]}$) given trough
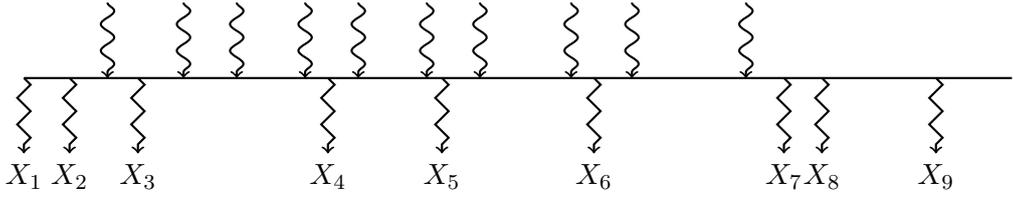
\begin{align*}
\begin{tikzpicture}[thick]
	\draw (0,0)--(13,0);
	\draw [->,decorate, decoration=snake] (1.1,1) -- (1.1,0) ;
	\draw [->,decorate, decoration=snake] (2.1,1) -- (2.1,0);
	\draw [->,decorate, decoration=snake] (2.8,1) -- (2.8,0);
	\draw [->,decorate, decoration=snake] (3.7,1) -- (3.7,0);
	\draw [->,decorate, decoration=snake] (4.4,1) -- (4.4,0) ;
	\draw [->,decorate, decoration=snake] (5.3,1) -- (5.3,0) ;
	\draw [->,decorate, decoration=snake] (6,1) -- (6,0) ;
	\draw [->,decorate, decoration=snake] (7.2,1) -- (7.2,0);
	\draw [->,decorate, decoration=snake] (8,1) -- (8,0) ;
	\draw [->,decorate, decoration=snake] (9.5,1) -- (9.5,0);
	\draw [->,decorate, decoration=zigzag] (0,0) -- (0,-1) node [below]{$X_1$};
	\draw [->,decorate, decoration=zigzag] (0.6,0) -- (0.6,-1) node [below]{$X_2$};
	\draw [->,decorate, decoration=zigzag] (1.5,0) -- (1.5,-1) node [below]{$X_3$};
	\draw [->,decorate, decoration=zigzag] (4,0) -- (4,-1) node [below]{$X_4$};
	\draw [->,decorate, decoration=zigzag] (5.5,0) -- (5.5,-1) node [below]{$X_5$};
	\draw [->,decorate, decoration=zigzag] (7.5,0) -- (7.5,-1) node [below]{$X_6$};
	\draw [->,decorate, decoration=zigzag] (10,0) -- (10,-1) node [below]{$X_7$};
	\draw [->,decorate, decoration=zigzag] (10.5,0) -- (10.5,-1) node [below]{$X_8$};
	\draw [->,decorate, decoration=zigzag] (12,0) -- (12,-1) node [below]{$X_9$};
\end{tikzpicture},
\end{align*}
where a snake arrow depicts an arrival, a zigzag arrow a departure and $X_i=(T_i,N(T_i))$. Note that the system is idle after the second and the third departure, respectively, and occupied apart from that.\\

We contrast above by depicting $i_{[1:9]}$ the same way.
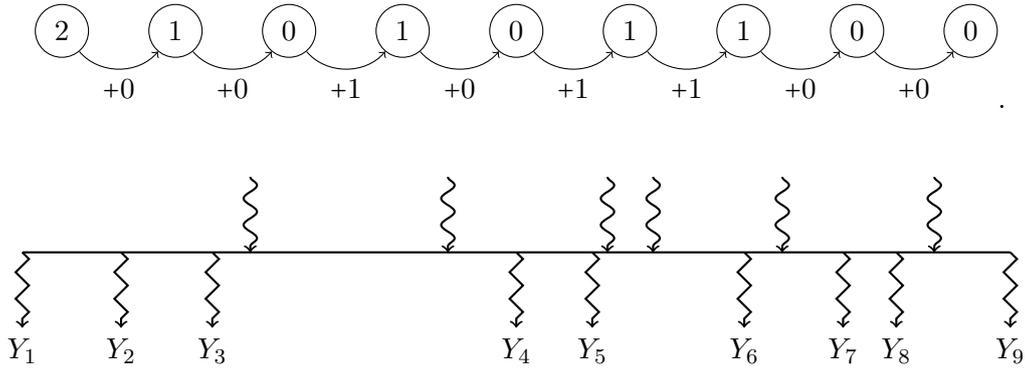
\begin{align*}
\begin{tikzpicture}[node distance=2em]
	\node[state, circle, inner sep=0pt, minimum size=0.7cm] (S1) {$2$};
	\node[state, right=of S1, circle, inner sep=0pt, minimum size=0.7cm] (S2) {$1$};
	\node[state, right=of S2, circle, inner sep=0pt, minimum size=0.7cm] (S3) {$0$};
	\node[state, right=of S3, circle, inner sep=0pt, minimum size=0.7cm] (S4) {$1$};
	\node[state, right=of S4, circle, inner sep=0pt, minimum size=0.7cm] (S5) {$0$};
	\node[state, right=of S5, circle, inner sep=0pt, minimum size=0.7cm] (S6) {$1$};
	\node[state, right=of S6, circle, inner sep=0pt, minimum size=0.7cm] (S7) {$1$};
	\node[state, right=of S7, circle, inner sep=0pt, minimum size=0.7cm] (S8) {$0$};
	\node[state, right=of S8, circle, inner sep=0pt, minimum size=0.7cm] (S9) {$0$};
	 \path[->]
    (S1) edge [bend right=50] node[below] {$+0$} (S2)
		(S2) edge [bend right=50] node[below] {$+0$} (S3)
		(S3) edge [bend right=50] node[below] {$+1$} (S4)
		(S4) edge [bend right=50] node[below] {$+0$} (S5)
		(S5) edge [bend right=50] node[below] {$+1$} (S6)
		(S6) edge [bend right=50] node[below] {$+1$} (S7)
		(S7) edge [bend right=50] node[below] {$+0$} (S8)
		(S8) edge [bend right=50] node[below] {$+0$} (S9);
\end{tikzpicture}.
\end{align*}
\begin{align*}
\begin{tikzpicture}[thick]
	\draw (0,0)--(13,0);
	\draw [->,decorate, decoration=snake] (3,1) -- (3,0) ;
	\draw [->,decorate, decoration=snake] (5.6,1) -- (5.6,0);
	\draw [->,decorate, decoration=snake] (7.7,1) -- (7.7,0);
	\draw [->,decorate, decoration=snake] (8.3,1) -- (8.3,0);
	\draw [->,decorate, decoration=snake] (10,1) -- (10,0) ;
	\draw [->,decorate, decoration=snake] (12,1) -- (12,0) ;
	\draw [->,decorate, decoration=zigzag] (0,0) -- (0,-1) node [below]{$Y_1$};
	\draw [->,decorate, decoration=zigzag] (1.3,0) -- (1.3,-1) node [below]{$Y_2$};
	\draw [->,decorate, decoration=zigzag] (2.5,0) -- (2.5,-1) node [below]{$Y_3$};
	\draw [->,decorate, decoration=zigzag] (6.5,0) -- (6.5,-1) node [below]{$Y_4$};
	\draw [->,decorate, decoration=zigzag] (7.5,0) -- (7.5,-1) node [below]{$Y_5$};
	\draw [->,decorate, decoration=zigzag] (9.5,0) -- (9.5,-1) node [below]{$Y_6$};
	\draw [->,decorate, decoration=zigzag] (10.8,0) -- (10.8,-1) node [below]{$Y_7$};
	\draw [->,decorate, decoration=zigzag] (11.5,0) -- (11.5,-1) node [below]{$Y_8$};
	\draw [->,decorate, decoration=zigzag] (13,0) -- (13,-1) node [below]{$Y_9$};
\end{tikzpicture}.
\end{align*}
We contrast $g_{[1:9]}$ and $i_{[1:9]}$ due to the following reasons. Having a look at the path of $g_{[1:9]}$ and $i_{[1:9]}$, respectively, one observes that the $i$-path has more idle times as the $g$-path. Roughly speaking, that is because the $g$-path is governed by more increments of higher magnitude. Thus, the probability that the system is unoccupied is lower. This responds to the fact that the arrival and the departure stream of $g$ do not look like having the same intensity rate, while the streams of $i$ rather do. However, this is a basic assumption we use, since the system is assumed to have run an infinitely long time. So, if $g_{[1:9]}$ would be a ``typical'' path, the embedded Markov chain is rather likely to be transient. Above examples are chosen to clarify the meaning of the statistic $\tau$ by means of as distinct numbers as possible. Further, keep in mind that only the stream below the horizontal line is communicated to us as data. 
Notice that the only equivalences among above data examples are
\begin{align*}
a_{[1:9]}\sim_{\tau}b_{[1:9]}\sim_{\tau}c_{[1:9]} \textit{ and }
d_{[1:9]}\sim_{\tau}e_{[1:9]}\sim_{\tau}f_{[1:9]}\sim_{\tau}g_{[1:9]}.
\end{align*}
\end{example}

\begin{remark}
\begin{enumerate}[(i)]
	\item In a sense, the statistic $\tau$ can be seen to lie in between of $t$ and $o$, $o$ denoting order statistic,  but closer to $t$ in the following sense. $t$-equivalence clearly implies $\tau$-equivalence, since one can recover the number of appearance of states among data, as well as the increments $\iota$. The converse is not true nor does $o$-equivalence imply $\tau$-invariance or the other way around, respectively. Of course, $t$-equivalence implies $o$-equivalence, see e.g. \citet[Proposition (27)]{diaconis1980finetti}.
	\item For a data string of finite length, it is obvious that the number of elements of the corresponding equivalence class is bounded. However, eliciting the exact number of elements included in this class seems to be an interesting but hard task, as it is in the case of $t$. We leave this as an open combinatorial problem.
\end{enumerate}
\end{remark}

We continue by mentioning an observation concerning the number of arrivals and departures that occur within a time horizon of observation $T_n-T_1$. By definition, it is clear that one observes $n-1$ departures. However, even if the statistic $\tau$ keeps track of the number of increments of the process, it may happen that the number of arrivals differ in $\tau$-equivalent strings. For instance note that in above example the number of arrivals in $b_{[1:9]}$ exceeds that of $a_{[1:9]}$ by one despite $a_{[1:9]} \sim_{\tau} b_{[1:9]}$.
We pin that fact as a proposition.

\begin{proposition}
The number of departing customers during $T_n-T_1$ is an invariant of $\tau$, while the number of arriving customers is not.
\end{proposition}

Now, we go on to shed more light on the aforementioned fact that the numbers of arrivals in equivalent strings can differ. This fact corresponds in a sense to the one that equivalent strings may not end with the same symbol. However, in the case of mixtures of Markov chains this is not true, i.e. the terminal state $x_n$ of a data string $x_{[1:n]}$ is completely determined by $x_1$ and the transition-counts. For a proof see e.g. \citet[Lemma 6.1.1.]{martin1967bayesian} and notice that it continues to hold true for countable infinite state spaces.

\begin{lemma}\thlabel{tau}
Let for $a_{[1:n]},b_{[1:n]}\in\D_n$ hold $a_{[1:n]}\sim_{\tau} b_{[1:n]}$. Then 
\begin{enumerate}[(i)]
  \item $a_n\in\{0,1\}$ $\Leftrightarrow$ $b_n\in\{0,1\}$,
	\item $a_n=r>1$ $\Leftrightarrow$ $b_n=r$.
\end{enumerate}
\end{lemma}
\begin{proof}
(i) By contradiction.\\
 Suppose $a_n\in\{0,1\}$ and $b_n\geq 2$. Since $a_{[1:n]}\sim_{\tau}b_{[1:n]}$, one necessarily has
\begin{align*}
\sum\limits_{r\in\N_0}\iota_r(a_{[1:n]})=\sum\limits_{r\in\N_0}\iota_r(b_{[1:n]}).
\end{align*}
Thus,
\begin{align*}
\sum\limits_{k=1}^{n-1}\left[ a_{k+1}-a_k+(1-\delta_{a_k0})\right]=\sum\limits_{k=1}^{n-1}\left[ b_{k+1}-b_k+(1-\delta_{b_k0})\right], \tag{$\ast$}
\end{align*}
which in turn yields
\begin{align}
a_n+\sum\limits_{k=1}^{n-1}\delta_{a_k 0}=b_n+\sum\limits_{k=1}^{n}\delta_{b_k 0}, \tag{$\ast \ast$}
\end{align}
since $a_1=b_1$ and $b_n\neq 0$ by assumption. \\
Now, treat the two possible cases $a_n\in\{0,1\}$ separately.\\

Case 1:($a_n=0$) By ($\ast \ast$), one has
\begin{align*}
&0-b_n=\sum\limits_{k=1}^n \delta_{b_k 0} -\sum\limits_{k=1}^n \delta_{a_k 0} +\underbrace{\delta_{a_n 0}}_{=1}\\
&\Rightarrow b_n=-1.
\end{align*}
Case 2:($a_n=1$) Again by ($\ast \ast$), one has
\begin{align*}
&1-b_n=\sum\limits_{k=1}^{n}\delta_{b_k 0} - \sum\limits_{k=1}^{n}\delta_{a_k 0}+\underbrace{\delta_{a_n 0}}_{=0}\\
&\Rightarrow b_n=1.
\end{align*}
This shows necessity. By interchanging the roles of $a_{[1:n]}$ and $b_{[1:n]}$, sufficiency is proven the same way. Thus, (i) follows.\\

(ii) To prove (ii), note that 
\begin{align*}
\sum\limits_{k=1}^{n-1}\delta_{a_k 0}=\sum\limits_{k=1}^{n}\delta_{a_k 0}=\sum\limits_{k=1}^{n}\delta_{b_k 0}=\sum\limits_{k=1}^{n-1}\delta_{b_k 0},
\end{align*}
exploiting (i) and $a_{[1:n]}\sim_{\tau}b_{[1:n]}$. 
Thus, ($\ast$) yields
\begin{align*}
a_n-a_1=b_n-b_1.
\end{align*}
\end{proof}

The lemma states that equality of the terminal state must only hold if the terminal state exceeds $1$. This reflects the fact that the laws over which one mixes are governed by $\Delta_{1,1}$-matrices. 
As a consequence of the \thref{tau} we have that the statistic $\tau$ possesses a certain kind of algebraic structure which reflects stability of the equivalence classes induced by the statistic with respect to extension of the data. This structure was discovered in \citet{freedman1962invariants} who called it $S$-structure. Here we consider its analog on the space $\D$. That is a statistic $\sigma=\left\{\sigma_n\right\}_{n\in\N}$ is said to have $S$-structure on $\D$ if for $a_{[1:n]},b_{[1:n]}\in\D_n$ and $x_{[1:m]},y_{[1:m]}\in\D_m$ such that $a_{[1:n]}\sim_{\tau}b_{[1:n]}$, $x_{[1:m]}\sim_{\tau}y_{[1:m]}$ and $a_{[1:n]}x_{[1:m]},b_{[1:n]}y_{[1:m]}\in\D_{n+m}$ it holds true that $a_{[1:n]}x_{[1:m]}\sim_{\tau}b_{[1:n]}y_{[1:m]}$, $n,m\in \N$. $S$-structure will then enable us to identify stationary measures summarized by $\tau$ in a unique way as mixtures of laws of embedded Markov chains of $M/G/1$.

\begin{proposition}
$\tau$ has $S$-structure on $\D$.
\end{proposition}
\begin{proof}
By above \thref{tau}, nothing has to be shown if $a_n>1$. However, if $a_n\in\{0,1\}$ and $a_n\neq b_n$ then, again by the lemma, one has $b_n=1-a_n$. Since the stacked data strings of length $n+m$ are down-skip-free by assumption, one has $a_{[1:n]}x_{[1:m]}\sim_{\tau}b_{[1:n]}y_{[1:m]}$. Indeed, observing an increment from $0$ to $r\in\N_0$ provides the same $\tau$-information as from $1$ to $r$ by definition of $\tau$.
\end{proof}

Notice that without recording the number of zeros among the data, one would not have $S$-structure. That is $\tilde{\tau}[a_{[1:n]}]:=(a_1,I(a_{[1:n]}))$ does not possess S-structure. For the sake of illustration take $g_{[1:9]}$ and $h_{[1:9]}$ from previous example and note that both have $1$ as initial state and four increments of magnitude $0$ and $2$, respectively. Hence, $g_{[1:9]}\sim_{\tilde{\tau}}h_{[1:9]}$ but $g_{[1:9]}44$ is not $\tilde{\tau}$-equivalent to $h_{[1:9]}44$. Roughly speaking, seeing how often the system is idle is informative for the $M/G/1$ system.\\

The following lemma states an invariance property of laws summarized by $\tau$ with respect to scaling the magnitude of the coordinate process. It is necessary for the proof of the subsequent theorem since it will establish the homogeneity of the stochastic matrix.

\begin{lemma}
Let $P\in\cP(\D)$ be summarized by $\tau$. Furthermore, let $x_{[1:n]}\in\D_n$ with $x_i>0$, $\forall i=1,\dots,n$ and for $r\geq 1$ let $y_{[1:n]}=x_{[1:n]}+r_{[1:n]}$, where $r_{[1:n]}=(r,r,\dots,r)$. Then it holds true that$P(x_{[1:n]}|x_1)=P(y_{[1:n]}|y_1)$.
\end{lemma}
\begin{proof}
First of all note that $y_{[1:n]}\in\D_n$ and that $\sum\limits_{i=1}^n\delta_{x_i 0}=0 \Leftrightarrow  \sum\limits_{i=1}^n \delta_{y_i 0}=0$. Moreover,
\begin{align*}
&P\left(x_{[1:n]}|x_1\right)=P\left(x_{[1:n]}|x_1,\sum\delta_{x_i 0}=0\right)=P\left(x_{[1:n]}|\tau_n(x_{[1:n]})\right)P\left(I(x_{[1:n]})|x_1, \sum\delta_{x_i 0}=0\right)\\
&=P\left(x_{[1:n]}|x_1,\sum\delta_{x_i 0}=0, (\#\{j:x_{j+1}-x_j=s,j=1,\dots,n\})_{s\in\N_0}\right)P\left(I(x_{[1:n]})|x_1, \sum\delta_{y_i 0}=0\right)\\
&=P\left(y_{[1:n]}|x_1+r,\sum\delta_{y_i 0}=0, (\#\{j:y_{j+1}-y_j=s,j=1,\dots,n\})_{s\in\N_0}\right)P\left(I(y_{[1:n]})|x_1+r, \sum\delta_{y_i 0}=0\right)\\
&=P\left(y_{[1:n]}|\tau_n(y_{[1:n]})\right)P\left(I(y_{[1:n]})|x_1, \sum\delta_{y_i 0}=0\right)=P\left(y_{[1:n]}|y_1\right).
\end{align*}
\end{proof}

We are now in position to state the mixing theorem for $M/G/1$, which will give rise to a prior distribution that is concentrated on the subspace of Markov measures (MM) that are governed by stochastic matrices which are of the shape discussed in the present section. Therefor let $\Delta_{1,1}^{(h)}$ denote the space of homogenous $1,1$ delta matrices. Moreover, let $MM[\Delta_{1,1}^{(h)}]$ be the space of Markov measures governed by those stochastic matrices and equip this sapce with the sigma field induced by weak convergence.

\begin{theorem}\thlabel{mix}
Let $P\in\cP(\D)$ be a shift-invariant probability summarized by $\tau$. Then, $P$ is a convex mixture of Markov measures governed by homogenous $\Delta_{1,1}$ stochastic matrices. That is, there is a unique measure $\mu\in \cP(\Delta_{1,1}^{(h)})$ such that 
\begin{align*}
P(\cdot)=\int\limits_{MM[\Delta_{1,1}^{(h)}]}Q(\cdot) \mu(dQ).
\end{align*}
\end{theorem}
\begin{proof}
By David Freedman's $S$-structure Theorem $[$\citet[Theorem 1]{freedman1962invariants}$]$ it holds that $P$ is a mixture of shift-ergodic laws being themselves summarized by $\tau$.
Since $\tau$ is a function of $t$, $P$ is a mixture of MM $[$\citet[Theorem 2]{freedman1962invariants}$]$. Moreover, the space of MM's supporting the mixing measure $\mu$ consists of laws governed by stochastic matrices possessing $\Delta$-shape since $supp(P)=\D$. By above lemma, homogeneity of these matrices follows for all row indexes $i\geq 1$. To see, that the zeroth row has to equal the first, just note that $0101\sim_{\tau}0110 \Rightarrow m_{01}=m_{11}$. Using this and  $01010\sim_{\tau}01100$ it follows $m_{00}=m_{10}$. Now, since $\tau$ has $S$-structure, $0101r\sim_{\tau}0110r$ for $r>1$. But then, $Q\in MM$ being summarized by $\tau$ yields $Q(0101r)=Q(0110r)$ and this, in turn, $m_{0r}=m_{1r}$.
\end{proof}

To state \thref{mix} in the language of Choquet theory $[$see e.g. \citet{phelps2001lectures}$]$, the space of stationary measures that are summarized by the statistic $\tau$ is a simplex with boundary consisting of all Markov measures governed by homogenous $\Delta_{1,1}$-matrices. Any nontrivial convex mixture thus gives a barycenter in the interior of this simplex. Using an obvious parametrization, one can state the result rather statistically.

\begin{corollary}\thlabel{mixcor}
Let $X_{[1:\infty]}$ be a sequence of stationary data with state space $\N_0$ inducing a joint distribution which is summarized by $\tau$. 
Then for $n\in\N$ and all strings of data $x_{[1:n]}$ one has
\begin{align*}
\bP(X_j=x_j ;j=1,\dotsc,n)=\int\limits_{\Delta_{1,1}^{(h)}} \nu_{x_1} \prod\limits_{i=1}^{n-1} m_{x_i,x_{i+1}} \tilde{\mu}(d\nu, dM).
\end{align*}
\end{corollary}

The corollary states that the problem of finding a prior distribution, modeling the mixing measure $\mu$, can be reduced to that of finding a random object which takes a.s. values in the space $\cP(\N_0)\times\Delta_{1,1}^{(h)}$ and whose distribution is analytically tractable. Obviously, the random objects $\nu$ and $M$ are dependent which complicates the model in general. However, since $\nu$ is the unique invariant distribution with respect to $M$, it is fully determined by $M$ and thus can be viewed as an injective function of $M$. This simplifies the mixture in \thref{mixcor} in the way that one merely has to take into account the distribution of the random stochastic $\Delta_{1,1}^{(h)}$ matrix. That is
\begin{align*}
\bP(X_j=x_j ;j=1,\dotsc,n)=\int\limits_{\Delta_{1,1}^{(h)}} \nu_{x_1}(M) \prod\limits_{i=1}^{n-1} m_{x_i,x_{i+1}} \hat{\mu}(dM),
\end{align*}
for a $\hat{\mu}$ suitably related to $\mu$.\\

It is known that measures summarized by a statistic can be described in ergodic theoretical terms. For instance, exchangeable probability measures, i.e. measures summarized by the order statistic, are invariant with respect to transformations induced by finite permutations. Furthermore, stationary measures summarized by transition counts can be argued to be invariant with respect to transformations induced by switching certain blocks. In both of these cases, there is a fact that simplifies the investigation of classes of equivalent strings, namely the multi-set of symbols through which a string passes is the same for all equivalent strings. However, in the situation considered here this is not necessarily true as above example shows. Hence, a description in terms of certain permutations of the corresponding sequence of increments seems more useful since their multi-sets are invariants for $\tau$. What can be certainly stated is that under $\tau$-equivalence invariance of the probability measure with respect to ``more'' than only block-switch transformations holds. It seems natural to investigate transformations induced by changes in the increments. In order to formalize this approach, denote for a string $x_{[1:n]}$ the string of ordered increments occurring in $x_{[1:n]}$ as $I[x_{[1:n]}]:=i_{[1:n-1]}(x_{[1:n]})$.

\begin{proposition}\thlabel{blockprop}
If $P$ is a stationary probability measure summarized by $\tau$, then $P$ is invariant with respect to the transformations induced by the following operations
\begin{enumerate}[(i)]
	\item switching two blocks whenever these have the same initial state and (a) end with the same symbol or (b) one ends with a $0$ and the other with a $1$,
	\item for a permutation $\sigma$ of $(k+1)$ elements, permuting a string of positive increments $\left(i_m, \dotsc,i_{m+k}\right)$ into $\left(i_{\sigma(m)},\dotsc, i_{\sigma(m+k)}\right)$.
\end{enumerate}
\end{proposition}
\begin{proof}
One has to argue that the value of the statistic $\tau$ remains the same under above transformations. Since an increment from $0$ has the same observable character than an increment from $1$, the assertion of (i) follows from \thref{tau}. For (ii) note that permuting the order of increments within a block of positive increments does not change the accumulated increment over this block. Thus, the number of $0$'s of the string remains the same.
\end{proof}

The transformations induced by (i) and (ii) of \thref{blockprop} give necessary conditions. However, these are not sufficient, i.e. for two strings that are equivalent with respect to $\tau$ it is in general not possible to turn one into the other by only applying transformations of the mentioned types. As an example regard the strings $2324321\sim_{\tau} 2354321$. Note that it is not possible to generate state $5$ in the first string by solely applying transformations of type (i) and (ii). The reason is that also block-switch transformations of the increments are allowed that keep the number of zeros among the string constant. However, one can hardly formalize those transformations in a neat way and should stick to the description using the statistic $\tau$.

A further description of measures being summarized by the statistic $\tau$ can be given using a common ergodic theoretical categorization. This again reflects the fact that $\tau$ lies, in a sense, in-between the order statistic and the transition counts. Therefor, keep in mind that all probabilities summarized by the order statistic are invariant under finite permutations, and note from \citet{diaconis1980finetti} that stationary probabilities summarized by transition counts are invariant with respect to a certain subgroup of permutations. This subgroup consists of all permutations that can be described as transformations of blocks that begin with the same symbol and end with the same symbol and hence do not affect the transition counts.\\

Having clarified the statistical structure of the data we are observing, we continue by finding a suitable prior, i.e. by modeling the mixing measure $\mu$ in above theorem. We motivate this modeling by an urn process. For urn processes yielding suitable prior distributions for Bayesian statistics see e.g. \citet{blackwell1973ferguson}, \citet{hoppe1984polya} , \citet{fortini2012hierarchical} and references therein.
However, our prior can not be chosen in a way such that the rows of the stochastic matrix are seen to be independently sampled. This is due to the following fact $[$see \citet[Corollary 1]{fortini2014predictive}$]$. The rows of $M$ are stochastically independent with respect to $\mu$ if and only if transition counts together with recording the first state are predictively sufficient, i.e. if and only if the probability of observing the next datum only depends on the observation of the last state and the observed number of transitions out of this last state. This clearly fails in our context since transitions of the same magnitude are informative no matter what the starting state of these transitions was. \\

Consider the following situation. Suppose there is a countable infinite set $C$ called the color space. Without loss of generality, take $C=\N_0$. Furthermore, suppose there is an urn $U_i$ associated to each color $i\in C$, i.e. think of $U_i$ being colored by color $i$. Let $U_i$ contain initially $\alpha_i$ black balls and start drawing a black ball from urn $U_{x_1}$, where $x_1$ is chosen according to a (stationary) start distribution $p_0\in\cP(\N_0)$. Then, having drawn the black ball, replace it together with a ball of color $x_2$ sampled by a color distribution $c_{x_1}\in\cP(\N_0)$ and move to urn $U_{x_2}$. Once a colored ball is sampled from an urn, replace it together with another ball of this color and move to the urn of this color. Otherwise, continue as before. This is the general definition of a reinforced Hoppe urn process. However, in the here considered situation slight modifications are needed. Proceed as described but with the following constraints. Firstly, the initial number of black balls is the same for all urns, i.e. $\alpha_i=\alpha$, for all $i\in\N_0$. Secondly, if one draws a ball from urn $i\in \N_0$, then the support of the color sampling distribution $c_i$ is shrunk to $supp(c_i)=\N_0 \backslash \{0,1,\dots,i-2\}$ for $i>1$ and $supp(c_i)=\N_0$ for $i=0,1$. Moreover, the color sampling distributions fulfill the shift condition $c_i(\{j\})=c_0(\{j-i+1\})$. Thirdly, not only the present urn is reinforced but all urns are reinforced the following way. If one draws from $U_i$ a ball of color $j$ then replace it together with an additional ball of the same color and add to $U_k$, $k\neq i$, an additional ball of color $j-i+1-\delta_{k0}-\delta_{i0}$ and move to $U_j$. If a black ball is drawn from $U_i$, sample a color $l$ and replace the black ball together with the ball of that color. Additionally, add a ball of color $l-i+1-\delta_{k0}-\delta_{i0}$ to $U_k$, $k\neq i$. Now, let $X_{[1:\infty]}$ denote the process of the colors successively sampled according to above urn process. That yields the predictive scheme
\begin{align*}
X_{n+1}=\bullet|X_{[1:n]} \sim \frac{\alpha}{\alpha+n-1} c_0(\{ \bullet-X_n+1\})+\frac{1}{\alpha+n-1} \sum\limits_{i=1}^{n-1} \delta_{X_{i+1}-X_i+(1-\delta_{X_i 0})}(\{\bullet\}). \tag{P}
\end{align*}
Now, it is shown in \citet{fortini2012hierarchical} that a reinforced Hoppe urn process as introduced above is partially exchangeable. Further, it is well known from \citet{blackwell1973ferguson} that the right hand side of (P) converges for $n\rightarrow \infty$ to a Dirichlet process with base measure $\alpha c_0(\cdot)$. However, this is essentially the same whatever value $X_n$ takes except of the different shifting of $c_0$ in (P). This motivates the following choice of a model for the prior $\mu$. \\

Let $\mu\in\cP(\Delta_{1,1}^{(h)})$ be the distribution on the space of homogenous $\Delta_{1,1}$ stochastic matrices such that the $0^{th}$ row of $M\in\Delta_{1,1}^{(h)}$ is sampled according to a Dirichlet process with base measure $\alpha c_0(\cdot)$ and the $i^{th}$ row, $i\geq 1$, is a copy of the $0^{th}$ row but shifted $i-1$ times to the right and the resulting ``empty'' entries of the row filled with zeros. Furthermore, (P) tells one how to update that prior distribution by seeing the data $X_{[1:n]}$. \\

To summarize the present section, assume that the data $X_{[1:\infty]}$ forms a stationary process with law summarized by $\tau$. Thus, $X_{[1:\infty]}$ is a mixture of stationary Markov chains governed by  homogenous $\Delta_{1,1}$ stochastic matrices. The distribution of the random stochastic matrix, i.e. the prior, is such that it makes rows dependently sampled from a Dirichlet process with parameters $\alpha>0$ and $c_0(\cdot)\in \cP(\N_0)$. Symbolically we write 
\begin{align*}
M&\sim Dir^{(\Delta)}(\alpha c_0)\\
X_{[1:\infty]}|M &\stackrel{\textsf{MM}}{\sim} M.
\end{align*}

The posterior of $M$ after having seen data $X_{[1:n]}$ is given by 
\begin{align*}
M|X_{[1:n]} \sim Dir^{(\Delta)}(c_n),
\end{align*}
where $c_n$ is the discrete measure given through
\begin{align*}
c_n(\{k\})=\alpha c_0(\{k\}) + \sum\limits_{i=1}^{n-1} \delta_{X_{i+1}-X_i+(1-\delta_{X_i 0})}(\{k\}).
\end{align*}

Thus, the posterior guess on the stochastic matrix $M$ is given by 
\begin{align*}
\E[M|X_{[1:n]}]=
\begin{pmatrix}
	\bar{c}_n(\{0\}) &\bar{c}_n(\{1\}) &\bar{c}_n(\{2\}) &\bar{c}_n(\{3\})&\bar{c}_n(\{4\})  &\dots\\
	\bar{c}_n(\{0\}) &\bar{c}_n(\{1\}) &\bar{c}_n(\{2\}) &\bar{c}_n(\{3\})&\bar{c}_n(\{4\})  &\dots\\
	0&\bar{c}_n(\{0\}) &\bar{c}_n(\{1\}) &\bar{c}_n(\{2\}) &\bar{c}_n(\{3\})&\ddots\\
	0&0&\bar{c}_n(\{0\}) &\bar{c}_n(\{1\}) &\bar{c}_n(\{2\}) &\ddots\\
	0&0&0&\bar{c}_n(\{0\}) &\bar{c}_n(\{1\}) &\ddots\\
	0&0&0&0&\bar{c}_n(\{0\}) &\ddots\\
	\vdots & \vdots & \vdots& \vdots& \ddots&\ddots
\end{pmatrix},
\end{align*}
where $\bar{c}_n(\{\cdot\})=\frac{c_n(\{\cdot\})}{\alpha+n-1}$. The prior guess is given similarly.

%%%%%%%%%%%%%%%%%%%%%%%%%%%%%%%%%%%%%%%%%%%%%%%%%%%%%

\section{Inference}\label{inference}
\subsection{Estimators for queueing characteristics}

In section \ref{prior} a posterior law of parametric form for the interdeparture time distribution was obtained. Moreover, a non-parametric posterior law for the stochastic matrix $M$ governing the embedded Markov chain of the $M/G/1$-system was obtained based on a Dirichlet process sampling the $0^{th}$ row of $M$, which describes the probability for the number of customers who enter the system during a service time. 
Now, we will use both to obtain an estimator for the service time distribution. Obtaining a direct and tractable closed-form prior to posterior analysis for the service time distribution based on observations given through the marked departure process as described before seems hardly possible. Instead a natural approach can be given by the connection provided by the functional relation ship
\begin{align*}
g(z)=a\left(1-\frac{z}{\lambda}\right).
\end{align*}
Exploiting this, we define a plug-in estimator for the service-time LST $g(\cdot)$ after having seen data $(N(T),T)_{[1:n]}$ by 
\begin{align*}
\hat{g}_n(z):=\gamma_n\left(1-\frac{z}{\bar{\lambda}_n}\right),
\end{align*}
where $\bar{\lambda}_n=\E_{\Gamma}[\lambda|(T_{i+1}-T_i)_{[1:n]}]$ denotes the posterior expected value of the variable $\lambda$ under the prior specified in the previous section and $\gamma_n(z)=\sum\limits_{k=0}^{\infty}z^k \bar{c}_n(\{k\})=\E_{\D^{(\Delta)}}\left[\sum\limits_{k=0}^{\infty}z^k A(\{k\})|X_{[1:n]} \right]$ denotes the posterior expected value of $a(\cdot)$, the p.g.f. of the discrete distribution of $A_S$, which was previously denoted as $A(\{\cdot\})$ and itself is regarded as being random. Notice that the interchange of the sum and the limit is justified since $\E_{\D}\left[a(z)\right]=\E_{\D}\left[\int_{\N_0} z^k P^{A_S}(dk)\right]$ and for any $z\in[0,1]$ the mapping $k\mapsto z^k$ is a real valued measurable function and $\int_{\N_0} z^k \alpha(dk)<\infty$ by assumption. Thus, see e.g. \citet{feigin1989linear} or \citet{phadia2015prior}, it holds that $\Pi_{\D}\left(\int_{\N_0}z^k P(dk) <\infty \right)=1$ and $\E_{\D}\left[\int_{\N_0} z^k P(dk)\right]=\int_{\N_0} z^k \E_{\D}[P](dk)$.\\

Based on the estimator $\hat{g}_n(\cdot)$, we are able to give estimators for other values of interest. One of those is the traffic intensity. The traffic intensity appears in further characteristics as the LST of the waiting-time distribution or the p.g.f. of the queue-length distribution and hence is of particular interest. An immediate approach is given by defining a plug-in estimator for $\rho$ through $\hat{\rho}_n:=\bar{\lambda}_n \hat{\sigma}_n$,
where $\hat{\sigma}_n$ is given by $\hat{\sigma}_n=-\left[\frac{\partial}{\partial z}\hat{g}_n(z)\right]_{|_{z=0}}$. Such estimators can be problematic with respect to translating (uniform) large sample results for $\gamma_n(\cdot)$ to that for $\hat{\sigma}_n$. However, in the here considered situation things become easier since $\rho$ has a direct relation to the random variable $A_S$ whose values we assume to observe. Indeed, one has 
\begin{align*}
\hat{\rho}=\hat{\lambda}_n \hat{\sigma}_n=\hat{\lambda}_n \left[\sum\limits_{k=0}^{\infty} k\left(1-\frac{z}{\hat{\lambda}_n}\right)^{k-1} \frac{1}{\hat{\lambda}_n} \bar{c}_n(\{k\})\right]_{|_{z=0}} =\sum\limits_{k=1}^{\infty} k \bar{c}_n(\{k\})=\E_{\E_{\D;n}[P^{A_S}]}\left[A_S\right].
\end{align*}

Furthermore, the estimators for $g(\cdot)$ and $\rho$ enable us to define estimators for other queueing characteristics as e.g. the waiting-time distribution, the busy-time distribution and sojourn-time distribution exploiting similar functional relationships. For details on these relationships see e.g. \citet[chapter 7]{nelson2013probability}. Define estimators for the following exact queueing characteristics
\begin{itemize}
	\item p.g.f. of number of customers in queue: $\hat{q}_n(z)=\frac{(1-\hat{\rho}_n)(1-z)}{\hat{g}_n(\bar{\lambda}_n(1-z))-z}$,
	\item p.g.f. of number of customers in the system: $\hat{m}_n(z)=\hat{g}_n(\bar{\lambda}_n(1-z)) \hat{q}_n(z)$,
	\item LST of waiting time of a customer in queue: $\hat{w}_n(s)=\frac{s(1-\hat{\rho}_n)}{s-\bar{\lambda}_n+\bar{\lambda}_n\hat{g}_n(s)}$,
\end{itemize}
while for the number of customers served in a busy period as well as the length of the busy period itself only estimates for the associated functional equation can be given, i.e.
\begin{itemize}
	\item LST of busy period: $b(s)=\hat{g}_n(s+\bar{\lambda}_n[1-b(s)])$,
	\item p.g.f. of number of customers served in a busy period: $m_b(z)=z\hat{g}_n(z)(\bar{\lambda}_n[1-m_b(z)])$.
\end{itemize}
Solutions to these equations may be understood as estimators for the busy time LST and the p.g.f. of the number of customers served in a busy period. However, the goodness of those estimators w.r.t. large samples is in question not only from a applied point of view, i.e. due to deviations appearing from numerical approximations, but also from a theoretical viewpoint since it is not known if minor changes in $\lambda$ and $\gamma$ do lead to minor changes in the solution to the equations. \\

We continue by emphasizing the role of the special form of the stochastic matrix $M$ with respect to the $M$-invariant distribution $p$ of the Markov chain $X_{[1:\infty]}$. We point out that the specific appearance of $M$ allows to write down explicitly the invariant distribution as a function of $M$ in form of its transform. That is, the diagram
\begin{center}
\[
\begin{xy}
  \xymatrix{
      M \ar[r]^{\phi} \ar[d]^*[@]{\simeq}_{\psi}    &   p \ar[d]^*[@]{\simeq}_{\tilde{\psi}}  \\          
      a(\cdot) \ar[r]_{\xi}             &   \pi(\cdot)  
  }
\end{xy}
\]
\end{center}
commutes. Therein, $\psi$ on the left-hand side denotes the composition of the mapping that extends the distribution of $A_S$ appearing in the $0^{th}$ row of $M$ to the whole of $M$ and the mapping that maps the distribution of $A_S$ onto its p.g.f., while on the right-hand side $\tilde{\psi}$ describes the mapping that maps the distribution $p$ onto its p.g.f. $\pi$.
Recall that the mapping $\xi$ is given as 
\begin{align*}
\xi:a(z)\mapsto a(z)\frac{(1-z)(1-a'(1))}{a(z)-z}=:\pi(z).
\end{align*}
Certainly, such a description is not possible in general. It even fails for the case of a non-homogenous $\Delta_{1,1}$ stochastic matrix which governs the embedded Markov chain of $M/G/1$ with state-dependent service, see \citet[equation (4)]{harris1967queues}. This special feature of standard $M/G/1$ enables us to give a direct estimator for the p.g.f. of the distribution of the system size at instants of departing customers which is, by the PASTA property, the same for any arbitrary instant of time. This estimator is given through
\begin{align*}
\hat{\pi}_n(z)=\gamma_n(z)\frac{(1-z)(1-\gamma_n'(1))}{\gamma_n(z)-z}.
\end{align*}

\subsection{Posterior consistency}

The estimators just defined are obvious ones, yet deserve some further theoretical justification. This will be given by posterior consistency which, roughly speaking, states that the mass of the posterior law will center around the true data-generating measure. To be more precise, let for a random probability measure $P\in \cP^{\Omega}$ a prior $\Pi\in\cP(\cP)$ be given. Further, let data $Y_{[1:\infty]}$ be given such that  $Y|P\stackrel{iid}{\sim} P$. Then, denote by $(\Pi_n)_{n\in\N_0}$ the sequence of posterior laws of $P$ given observed data $Y_{[1:n]}$, i.e. $\Pi_n(C)=\Pi(P\in C|Y_{[1:n]})$ and for the sake of completeness $\Pi_0:=\Pi$, for all sets $C$ in the sigma field induced by weak convergence of measures. The sequence $(\Pi_n)_{n\in\N_0}$ is called consistent at the true data-generating distribution $P_0$ if for $P_0$-almost all data sequences it holds that $\Pi_n\stackrel{w, n\rightarrow \infty}{\longrightarrow}\delta_{P_0}$.
First of all we state posterior consistency of the parametric sequence of posteriors for the interarrival rate $\lambda$.
\begin{lemma}\thlabel{rateconsistency}
For almost all sequences $T_{[1:\infty]}$ and for any $\epsilon>0$ it holds that 
\begin{align*}
\Pi_{\Gamma;n}([\lambda_0-\epsilon,\lambda_0+\epsilon])\stackrel{n\rightarrow \infty}{\longrightarrow} 1,
\end{align*}
where $\Pi_{\Gamma}$ denotes the prior distribution for $\lambda$ as specified in section \ref{prior} and $\lambda_0$ the true interarrival rate.
\end{lemma}
\begin{proof}
Let $\{Di\}_i$, $D_i:=T_{i+1}-T_i$ be the exponentially distributed interdeparture time data. By conjugacy of the gamma-distribution with respect to exponential likelihoods and well known properties of the gamma distribution, the posterior expected value of the arrival rate is given by 
\begin{align*}
\E_{\Gamma}[\lambda|D_{[1:n]}]=\frac{a+n}{b+\sum_{i=1}^{n}\left[D_i\right]},
\end{align*}
where $(a,b)\in\R_+^2$ are the prior parameters. Moreover, the posterior variance is given by
\begin{align*}
\mathbb{V}_{\Gamma}[\lambda |D_{[1:n]}]=\frac{a+n}{\left(b+\sum_{i=1}^{n}\left[D_i\right]\right)^2}.
\end{align*}
Thus, by means of the SLLN and the continuous mapping theorem, $\E_{\Gamma}[\lambda|D_{[1:n]}]\stackrel{n\rightarrow \infty}{\longrightarrow}1/\lambda_0$ and $\mathbb{V}_{\Gamma}[\lambda |D_{[1:n]}]\stackrel{n\rightarrow \infty}{\longrightarrow}0$ such that the assertion of the lemma follows from a straight forward application of the Markov inequality.
\end{proof}

Next, we study the posterior consistency of the random stochastic matrix $M\in \left[\Delta_{1,1}^{(h)}\right]^{\Omega} \subset\kS^{\Omega}$. Virtually, this task requires an extended definition of posterior consistency. So, let $M$ be a random matrix and let $Y_{[1:\infty]}$ be a Markov chain with countable state space that, given $M$, is governed by $M$, i.e. $Y_{[1:\infty]}\mid M\stackrel{MC}{\sim} M$. Let a prior $\Pi$ be given for $M$ and, as before, denote by $(\Pi(\cdot|Y_{[1:n]}))_{n\in\N_0}=((\Pi_n(\cdot))_{n\in\N_0}$ the sequence of posterior laws of $M$. Call $(\Pi_n)_{n\in\N_0}$ consistent if for $M_0$-almost all sequences of data $Y_{[1:\infty]}$ it holds that $\Pi_n(C_0)\stackrel{n\rightarrow \infty}{\longrightarrow}1$, for all sets $C_0$ in the sigma field on $\kS$ induced by coordinate-wise convergence containing $M_0$, the true stochastic matrix governing the data. Here, $M_0$-almost all sequences of data means the smallest set of data strings which has full mass under the stationary Markov probability measure that is induced by $M_0$. Next, we show the posterior consistency of the random matrix $M$.

\begin{lemma}\thlabel{matrixconsistency}
For almost all sequences of data $N(T)_{[1:\infty]}$ and all measurable neighborhoods $C_0$ of the true stochastic matrix $M_0$ governing the embedded Markov chain of $M/G/1$ it holds for the prior $\Pi_{\D^{(\Delta)}}$ specified in section 3 that
\begin{align*}
\Pi_{\D^{\Delta};n}(C_0)\stackrel{n \rightarrow \infty}{\longrightarrow} 1.
\end{align*}
\end{lemma}
\begin{proof}
Since $M_0\in \Delta_{1,1}^{(h)}$, it suffices to regard all neighborhoods of $M_0$ contained in the trace sigma field induced by $\Delta_{1,1}^{(h)}$. But then, using mapping $\psi$ in above diagram, it is enough to show consistency for the $0^{th}$ row of $M$, which is nothing but the distribution of the variable $A_S$. Since the posterior, emerging from the Dirichlet process prior updated in a manner as described before by data $X_{[1:n]}$,  is as well a Dirichlet process with updated base measure 
\begin{align*}
c_n(\{k\})= \alpha c_0(\{k\})+ \sum\limits_{i=1}^{n-1}\delta_{X_{i+1}-X_i+(1-\delta_{X_i 0})}(\{k\}),
\end{align*}
posterior consistency of $M$ follows from convergence properties of that prior process, see e.g. \citet[chapter 3]{ghosh2003bayesnonp}.
\end{proof}

The posterior consistency of the random matrix immediately yields consistency of the Bayes estimator for the stochastic matrix and for the p.g.f. of $A_S$, respectively.

\begin{corollary}\thlabel{matrixcorollary}
For $M_0$-almost all sequences of data $X_{[1:n]}:=N(T)_{[1:n]}$, it holds that 
\begin{enumerate}[(i)]
	\item $\E_{\D^{\Delta}}\left[M\mid X_{[1:n]} \right]\stackrel{n\rightarrow \infty}{\longrightarrow} M_0$,
	\item $\textsf{Var}_{\D^{\Delta}}\left[M\mid X_{[1:n]} \right]\stackrel{n\rightarrow \infty}{\longrightarrow} 0$,
	\item for any $\tau>0$, $\sup\limits_{z\in[0,\tau]} \left|\gamma_n(z)-a_0(z)\right|\stackrel{n\rightarrow \infty}{\longrightarrow} 0$.
\end{enumerate}
\end{corollary}
\begin{proof}
The assertions of (i) and (ii) just follows as necessary consequences of \thref{matrixconsistency}. For (iii) note that one has
\begin{align*}
&\lim\limits_{n\rightarrow\infty}\sup\limits_{z\in[0,\tau]} \left|\gamma_n(z)-a_0(z)\right|\leq \lim\limits_{n\rightarrow\infty} \sup\limits_{z\in[0,\tau]} \sum\limits_{k=0}^{\infty} z^k \left|\bar{c}_n(\{k\})-P_0^{A_S}(\{k\})\right|\\
&\leq \lim\limits_{n\rightarrow\infty} \sum\limits_{k=0}^{\infty}\tau^k \left|\bar{c}_n(\{k\})-P_0^{A_S}(\{k\})\right|
\leq  \lim\limits_{n\rightarrow\infty} \sum\limits_{k=0}^{\infty} \tau^k \E_{\D^{\Delta}}\left[\left|P^{A_S}(\{k\})-P_0^{A_S}(\{k\})\right| \bigg\mid X_{[1:n]}\right],
\end{align*}
such that the assertion of (iii) follows from (i) and the monotone convergence theorem.
\end{proof}

So far we established posterior consistency with respect to the direct observables. Now we show that the indirect estimator defined above possesses certain consistency properties as well. Recall that the LST of the service time distribution is expressed in terms of the p.g.f. of the distribution $A_S$ which in turn is a power series. Thus, having on mind techniques from complex analysis, it seems natural to undertake the investigation of posterior consistency within a framework that reflects this analytic approach. Hence, posterior consistency of the LST of the service time distribution will be stated as a kind of a.s. compact convergence inside the posterior law. Call a series of functions compact convergent to a limit function if its restriction to any compact set converges uniformly.
Therefor, let $g_0(\cdot)$ denote the the LST of the true service time distribution $G_0(\cdot)$, i.e. $g_0(z)=\int_0^{\infty} e^{-zs}dG_0(s)$.

\begin{theorem}
For almost all data sequences $(N(T),T)_{[1:n]}$, all $R>0$ and all $\epsilon>0$ it holds true that
\begin{align*}
\bP\left(\sup\limits_{z\in[0,R]} \left|\hat{g}_n(z)-g_0(z)\right| \geq \epsilon \bigg\mid (N(T),T)_{[1:n]}\right)\stackrel{n\rightarrow \infty}{\longrightarrow} 0.
\end{align*}
\end{theorem}
\begin{proof}
Let $R>0$ and $\epsilon>0$ be arbitrarily chosen real numbers. Define
\begin{align*}
&X:=\sup\limits_{z\in[0,R]} \left|\sum\limits_{k=0}^{\infty} \left(1-\frac{z}{\bar{\lambda}_n}\right)^k \bar{c}_n(\{k\}) - \sum\limits_{k=0}^{\infty}  \left(1-\frac{z}{\bar{\lambda}_n}\right)^k A_0(\{k\})\right|,\\
&Y:=\sup\limits_{z\in[0,R]} \left|\sum\limits_{k=0}^{\infty} \left(1-\frac{z}{\bar{\lambda}_n}\right)^kA_0(\{k\})-\sum\limits_{k=0}^{\infty} \left(1-\frac{z}{\lambda_0}\right)^kA_0(\{k\})\right|
\end{align*}
Then one has
\begin{align*}
&\bP\left(\sup\limits_{z\in[0,R]} \left|\hat{g}_n(z)-g_0(z)\right| \geq \epsilon \bigg\mid (N(T),T)_{[1:n]}\right)\\
&=\bP\left(\sup\limits_{z\in[0,R]} \left|\gamma\left(1-\frac{z}{\bar{\lambda}_n}\right)-a_0\left(1-\frac{z}{\lambda_0}\right)\right|\geq \epsilon \bigg\mid (N,T(N))_{[1:n]} \right)\\
&\leq \bP\left(X+Y\geq \epsilon, Y\geq \epsilon/2| (N,T(N))_{[1:n]}  \right)+ P\left(X+Y\geq \epsilon , Y< \epsilon/2 \mid (N,T(N))_{[1:n]}  \right)\\
&\leq \bP\left( Y\geq \epsilon/2| (N,T(N))_{[1:n]}  \right)+ P\left(X\geq \epsilon/2 \mid (N,T(N))_{[1:n]}  \right).
\end{align*}
Exploiting the independence assumption between $\lambda$ and $M$, for the first addend it follows
\begin{align*}
 &\bP\left( Y\geq \epsilon/2 \mid (N,T(N))_{[1:n]}  \right)\\
& \leq \Pi_{\Gamma}\left(\sum\limits_{k=0}^{\infty} A_0(\{k\}) \sum\limits_{i=0}^k  R^{k-i}\left|\bar{\lambda}_n^{-(k-i)}-\lambda_0^{-(k-i)}\right|\geq\epsilon/2 \bigg\mid T_{[1:n]}\right),
\end{align*}
while for the second one has
\begin{align*}
 &\bP\left(X\geq \epsilon/2 \mid (N,T(N))_{[1:n]}  \right)\\
&\leq  \bP\left( \sum\limits_{k=0}^{\infty}\sup\limits_{z\in[0,R]} \left|1-\frac{z}{\bar{\lambda}_n}\right|^k \left| \bar{c}_n(\{k\})-A_0(\{k\}) \right|\geq \epsilon/2 \bigg\mid (N(T),T)_{[1:n]} \right)\\
&\leq  \Pi_{\D^{(\Delta)}}\left( \sum\limits_{k=0}^{\infty} \left(1+\frac{R}{\lambda_0-O(n^{-\kappa})}\right)^k \left| \bar{c}_n(\{k\})-A_0(\{k\}) \right|\geq \epsilon/2 \bigg\mid N(T)_{[1:n]} \right),
\end{align*}
for some suitably chosen $\kappa>0$. Hence the assertion of the theorem follows from \thref{rateconsistency} and \thref{matrixcorollary}.
\end{proof}

As an immediate consequence, one has the a.s. uniform convergence of the estimator $\hat{g}_n(z)$ on sets of the form $\{z\in \R_+ : z\leq R\}$ for some positive real number $R$.

\begin{theorem}
For almost all data sequences $(N(T),T)_{[1:n]}$ and all $R>0$ it holds true that
\begin{align*}
\sup\limits_{z\in[0,R]} \left|\hat{g}_n(z)-g_0(z)\right|\stackrel{n\rightarrow\infty}{\longrightarrow}0.
\end{align*}
\end{theorem}
\begin{proof}
Let $R>0$ be an arbitrarily fixed positive real number. Then one has
\begin{align*}
&\sup\limits_{z\in[0,R]} \left|\hat{g}_n(z)-g_0(z)\right|=\sup\limits_{z\in[0,R]} \left|\gamma_n\left(1-\frac{z}{\bar{\lambda}_n}\right)-a_0\left(1-\frac{z}{\lambda_0}\right)\right|\\
&=\sup\limits_{z\in[0,R]} \left|\sum\limits_{k=0}^{\infty}\left(1-\frac{z}{\bar{\lambda}_n}\right)^k \bar{c}_n(\{k\})-\sum\limits_{k=0}^{\infty}\left(1-\frac{z}{\lambda_0}\right)^k A_0(\{k\})\right|\\
&\leq\sup\limits_{z\in[0,R]} \sum\limits_{k=0}^{\infty}\left|\left(1-\frac{z}{\bar{\lambda}_n}\right)^k \bar{c}_n(\{k\})-\left(1-\frac{z}{\lambda_0}\right)^k A_0(\{k\})\right|\\
&\leq \sum\limits_{k=0}^{\infty}\sup\limits_{z\in[0,R]}\left|1-\frac{z}{\bar{\lambda}_n}\right|^k \left|\bar{c}_n(\{k\})-A_0(\{k\})\right|+ \sum\limits_{k=0}^{\infty}A_0(\{k\}) \sum\limits_{i=0}^k R^{k-i}\left|\bar{\lambda}_n^{-(k-i)}-\lambda_0^{-(k-i)}\right|.
\end{align*}
Hence, the assertion of the theorem follows using above lemmas in combination with monotone convergence and continuous mapping theorems.
\end{proof}

Next, we point out that similar consistency properties hold for several derivatives of $\hat{g}_n(\cdot)$ as well as for other queueing characteristics mentioned earlier. 

\begin{theorem}\thlabel{thm}
Let $f\in \{w, q, m\}$, $\hat{f}_n(z)$ be one of the estimators defined at the beginning of this section and $f_0(z)$ the true transform. Then, for any $\tau>0$ and $\epsilon$, one has 
\begin{align*}
\bP\left(\sup\limits_{0\leq z \leq \tau} \left|\hat{f}_n(z)-f_0(z)\right|>\epsilon \mid (N(T),T)_{[1:n]} \right)\stackrel{n\rightarrow\infty}{\longrightarrow}0.
\end{align*}
Moreover, it holds
\begin{align*}
\sup\limits_{0\leq z \leq \tau}  \left|\hat{f}_n(z)-f_0(z)\right| \stackrel{n\rightarrow\infty}{\longrightarrow}0.
\end{align*}
\end{theorem}
\begin{proof}
The technical details of the proof are omitted. It is enough to apply above results together with the continuous mapping theorem analogously as in \citet{rohrscheidt2017bayesian}.
\end{proof}

\subsection{Posterior normality}

Having obtained the centering of the posterior law one might also be interested in getting an idea of the shape of the limiting posterior. This leads to so called Bernstein-von Mises type results.
Roughly speaking, those describe that the posterior law of the object of interest, centered at its estimate and rescaled suitably, looks more and more like a Gaussian distribution. Results of that kind are useful for simulations and give first insight into convergence rates of the posterior. The first nonparametric Bernstein-von Mises type result was obtained in \citet{conti1999large} where the author has proven that the posterior law of the p.g.f. of a random law drawn according to a Dirichlet process and centered at its Bayesian estimate converges towards a centered Gaussian process possessing a certain covariance structure. To be more precise, in the notation of the present work, it was obtained that under suitable constraints it holds that
\begin{align*}
\cL\left[\sqrt{n}\left[a(z)-\gamma_n(z)\right]|X_{[1:n]}\right]\stackrel{n\rightarrow \infty}{\longrightarrow} \cL\left[X(z)\right],
\end{align*}
where $X(\cdot)$ is a centered Gaussian process with covariance structure $H(u,v)=a_0(uv)-a_0(u)a_0(v)$. Weak convergence, thereby, is considered on the space of continuous functions equipped with the sup-norm and the theorem holds for almost all data sequences $X_{[1:\infty]}$. For the proof, Conti employed results from \citet{freedman1963asymptotic}. Moreover, it is easy to show in the parametric situation of the departure rate that for almost all data sequences $T_{[1:\infty]}$ it holds that
\begin{align*}
\cL\left[\sqrt{n}\left[\lambda-\bar{\lambda}_n\right]|T_{[1:n]}\right]\stackrel{n \rightarrow \infty}{\longrightarrow} \mathcal{N}(0,\lambda_0^{-2}).
\end{align*}
Thus, combining these two results, one has
\begin{theorem}
For almost all data sequences $(N(T),T)_{[1:n]}$ it holds that 
\begin{align*}
\cL\left[\sqrt{n}\left[g(z)-\hat{g}_n(z)\right]| (N(T),T)_{[1:n]} \right]\stackrel{n\rightarrow\infty}{\longrightarrow}\cL\left[G(z)\right],
\end{align*}
on the space of continuous functions equipped with the sup-norm. Here, $G(z)$ is a centered Gaussian process with covariance structure 
\begin{align*}
K(u,v)= H\left(1-\frac{u}{\lambda_0},1-\frac{v}{\lambda_0}\right)+uv \lambda_0^{-6} a'_0\left(1-\frac{u}{\lambda_0}\right)a'_0\left(1-\frac{v}{\lambda_0}\right).
\end{align*}
\end{theorem}
\begin{proof}
Instead of presenting the technicalities we refer the interested reader to \citet[Theorem 3]{conti1999large} and \citet[Section 5]{rohrscheidt2017bayesian} since the proof works along these lines taking into account results presented earlier in the present paper.
\end{proof}

Applying this result, one is able to give similar results for the centering of the estimators $\hat{f}_n$ appearing in \thref{thm}. In order to do so, the main work to do is to apply previous results and subsequently calculate the particular covariance structure. However, since these are rather non-telling calculations, we omit the details.\\

%%%%%%%%%%%%%%%%%%%%%%%%%%%%%%%%%%%%%%%%%%%%%%%%%%%%%
\vfill

\textbf{Funding:} This work was supported by the Deutsche Forschungsgemeinschaft (German Research
Foundation) within the programme \textit{Statistical Modeling of Complex Systems and
Processes---Advanced Nonparametric Approaches}, grant GRK 1953.

%%%%%%%%%%%%%%%%%%%%%%%%%%%%%%%%%%%%%%%%%%%%%%%%%%%%%

\bibliographystyle{humannat}
\bibliography{myliterature}

%%%%%%%%%%%%%%%%%%%%%%%%%%%%%%%%%%%%%%%%%%%%%%%%%%%%%

\bigskip

\begin{figure}[htbp]
\begin{minipage}{0.6\textwidth} 
	Cornelia Wichelhaus\\
	Technische Universit{\"a}t Darmstadt\\
	Schlossgartenstra{\ss}e 7\\
  64289 Darmstadt\\
	Germany\\
	wichelhaus@mathematik.tu-darmstadt.de
	\end{minipage}
	\hfill
	\begin{minipage}{0.6\textwidth}
	Moritz von Rohrscheidt\\
	Ruprecht-Karls Universit{\"a}t Heidelberg\\
	Berliner Stra{\ss}e 41-49\\
  69120 Heidelberg\\
	Germany\\
	rohrscheidt@uni-heidelberg.de
	\end{minipage}
\end{figure}

%%%%%%%%%%%%%%%%%%%%%%%%%%%%%%%%%%%%%%%%%%%%%%%%%%%%%

\end{document}